\documentclass[11pt]{article}

\usepackage{amssymb, amsmath, amsthm, graphicx, xcolor}
\usepackage[left=1in,top=1in,right=1in]{geometry}
\usepackage{subfigure}

\date{}

\theoremstyle{plain}
      \newtheorem{theorem}{Theorem}[section]
      \newtheorem{lemma}[theorem]{Lemma}
            \newtheorem{claim}[theorem]{Claim}
      
      \newtheorem{observation}[theorem]{Observation}
      \newtheorem{corollary}[theorem]{Corollary}
      
      \newtheorem{conjecture}[theorem]{Conjecture}
\theoremstyle{definition}

\theoremstyle{remark}

\title{On short edges in complete topological graphs}
\author{Andrew Suk\thanks{Department of Mathematics, University of California at San Diego, La Jolla, CA, 92093 USA. Supported by NSF CAREER award DMS-1800746 and  NSF award DMS-1952786. Email: {\tt asuk@ucsd.edu}.}}

\begin{document}

\maketitle

\begin{abstract} 
Let $h(n)$ be the minimum integer such that every complete $n$-vertex simple topological graph contains an edge that crosses at most $h(n)$ other edges.  In 2009, Kyn\v{c}l and Valtr showed that $h(n) = O(n^2/\log^{1/4} n)$, and in the other direction, gave constructions showing that $h(n) = \Omega(n^{3/2})$.  In this paper, we prove that $h(n) = O(n^{7/4})$.  Along the way, we establish a new variant of Chazelle and Welzl's matching theorem for set systems with bounded VC-dimension, which we believe to be of independent interest.

\end{abstract}

\section{Introduction}

A \emph{topological graph} is a graph drawn in the plane such that its vertices are represented by points
 and its edges are represented by non-self-intersecting arcs connecting the corresponding points. The edges are
allowed to intersect, but they may not intersect vertices except for
their endpoints.  Furthermore, no two edges are tangent, i.e., if two edges share an interior point in common, then they must properly cross at that point in common.  Two edges of a topological graph \emph{cross} if their interiors share a point, and are \emph{disjoint} if they have no points in common.  A topological graph is \emph{simple} if every pair of its edges intersect at most once,
either at a common endpoint or at a proper crossing point. If the edges are drawn
as straight-line segments, then the graph is said to be \emph{geometric}. 
  If the vertices
of a geometric graph are in convex position, then it is called \emph{convex}.  Finally, given a complete $n$-vertex topological graph $G$, we say that an edge in $G$ is \emph{short} if at most $o(n^2)$ other edges cross it.

It is easy to see that there are many complete topological graphs with no short edges.  Moreover,  Pach and T\'oth \cite{pachtoth} showed that, for every $n \geq 2$, there are complete $n$-vertex topological graphs with the property that every pair of edges has exactly one or two points in common, and hence, has no short edges. However, the situation changes for \emph{simple} topological graphs. Let $h(n)$ be the minimum integer such that every complete $n$-vertex simple topological graph contains an edge that crosses at most $h(n)$ other edges.  In 2003, Brass, Moser, and Pach \cite{brass} conjectured that one can always find a short edge in a complete simple toplogical graph, that is, $h(n)  = o(n^2)$.  The best known lower bound at the time was $h(n)  =  \Omega(n)$, due to a result of Harborth and Th\"urmann \cite{hat} from 1994. Their conjecture was verified by Kyn\v{c}l and Valtr \cite{kyncl} in 2009, who showed that  $$\Omega\left(n^{3/2}\right)\leq h(n) \leq O\left(\frac{n^2}{\log^{1/4}n}\right).$$ See also \cite{zeng} for a slightly better upper bound.  Our main result is the first polynomial improvement in the upper bound for $h(n)$.

\begin{theorem}\label{main}
For $n\geq 2$, we have $h(n) =O\left(n^{7/4}\right)$.
\end{theorem}

Our paper is organized as follows.  In the next section, we recall the notion of VC-dimension and prove a result on matchings with low stabbing number in set systems with bounded VC-dimension.  In Section \ref{p1}, we use this matching result to prove Theorem \ref{main}.  We conclude the paper with several remarks and open problems. For sake of clarity, we omit floor and ceiling signs whenever they are not crucial.

\section{VC-dimension and the dual shatter function}

Given a set-system $\mathcal{F}$ with a ground set $V$, the \emph{Vapnik-Chervonenkis dimension} (in short, VC-dimension) of $\mathcal{F}$ is the largest integer
$d$ for which there exists a $d$-element set $X\subset V$ such that for every subset $Y \subset X$, one can find a
member $A \in \mathcal{F}$ with $A\cap X = Y$.  The VC-dimension, introduced by Vapnik and Chervonenkis \cite{vc}, is one of the most useful combinatorial
parameters that measures the complexity of a set system, with a wide range of applications in statistics, logic, learning theory, graph drawing, and computational geometry. Throughout this paper, all set systems will be \emph{multisets}, that is, $\mathcal{F}$ may contain several copies of the same set $A \in \mathcal{F}$.  The cardinality of $\mathcal{F}$, denoted by $|\mathcal{F}|$, is counted with these multiplicities.  

Here, it will be more convenient to work with the dual VC-dimension and the dual shatter function of $\mathcal{F}$.  The \emph{dual shatter function} of $\mathcal{F}$, denoted by $\pi_{\mathcal{F}}^{\ast}(m)$, 
is the maximum number of nonempty cells of the Venn diagram of $m$ sets in $\mathcal{F}$.  In other words, $\pi_{\mathcal{F}}^{\ast}(m)$ is the maximum number of equivalence classes on $V$ defined by an $m$-element subfamily $\mathcal{S}\subset \mathcal{F}$, where two points $x, y \in V$ are equivalent with respect to $\mathcal{S}$ if $x$ belongs to the same sets of $\mathcal{S}$ as $y$ does. 
 The \emph{dual VC-dimension} of $\mathcal{F}$ is the maximum possible number of sets in $\mathcal{F}$ with a complete Venn diagram, that is, the maximum $d^{\ast}$ such that $\pi^{\ast}_{\mathcal{F}}(d^{\ast}) = 2^{d^{\ast}}$.  The classical theorem of Sauer \cite{sauer} and Shelah~\cite{shelah} implies that if $\mathcal{F}$ has dual VC-dimension $d^{\ast}$, then $\pi^{\ast}_{\mathcal{F}}(m) \leq \sum_{i = 0}^{d^{\ast}}\binom{m}{i}$.  Moreover, it is know that $d^{\ast} \leq 2^d$ if $\mathcal{F}$ has VC-dimension $d$ (see \cite{welzl}).

 Let $\mathcal{F}$ be a set system with ground set $V$. Given two points $u,v \in V$, we say that a set $A\in \mathcal{F}$ \emph{stabs} the pair  $\{u,v\}$ if $A$ contains exactly one element of $\{u,v\}$.  We say that the subset $X\subset V$ is \emph{$\delta$-separated} if for any two points $u,v \in X$, there are at least $\delta$ sets in $\mathcal{F}$ that stabs the pair $\{u,v\}$.  In 1995, Haussler proved the following interesting theorem.

\begin{lemma}[\cite{ha}]\label{dual}

Let $\mathcal{F}$ be a set system on a ground set $V$ with $\pi^{\ast}_{\mathcal{F}}(m) = c m^d$.  If $X\subset V$ is $\delta$-separated, then $|X| \leq c_1(|\mathcal{F}|/\delta)^{d}$ where $c_1= c'(d,c)$.

\end{lemma}

As an application of Lemma \ref{dual}, Fox, Pach, and the author \cite{fps} showed the following.

\begin{lemma}[\cite{fps}]\label{partition}

Let $\mathcal{F}$ be a set system on a ground set $V$ such that $|V| = n$ and $\pi^{\ast}_{\mathcal{F}}(m) = c m^d$.  Then there is a constant $c_2 = c_2(d,c) > 1$ such that for any $\delta$ satisfying $1 \leq \delta
\leq |\mathcal{F}|$, there is a partition $V = S_1\cup \cdots \cup S_r$ of $V$ into
$r \leq c_2(|\mathcal{F}|/\delta)^{d}$ parts, such that any pair of vertices from the same part $S_i$ is stabbed by at most $2\delta$ members of $\mathcal{F}$.  

\end{lemma}

\noindent Together with the pigeonhole principle, Lemma \ref{partition} implies the following.

\begin{corollary}[\cite{fps}]\label{pigeon}

Let $\mathcal{F}$ be a set system on a ground set $V$ such that $|V| = n$ and $\pi^{\ast}_{\mathcal{F}}(m) = c m^d$.  Then there is a constant $c_2 = c_2(d,c)$ such that for any $\delta$ satisfying $1 \leq \delta
\leq |\mathcal{F}|$, there is a subset $U\subset V$ such that $|U| \geq \frac{n}{c_2(|\mathcal{F}|/\delta)^{d}}$, and any pair of vertices in $U$ is stabbed by at most $2\delta$ members of $\mathcal{F}$.
\end{corollary}

In \cite{welzl}, Chazelle and Welzl proved that if $\mathcal{F}$ is a set system with ground set $V$, such that $|V| = n$, $n$ is even, and $\pi^{\ast}_{\mathcal{F}}(m)  = O(m^d)$, then there is a perfect matching $M$ on $V$ with the property that for any set $A \in \mathcal{F}$, $A$ stabs at most $O(n^{1 - 1/d})$ matchings in $M$.  One of the main tools we will need for the proof of Theorem \ref{main} is the following variant of Chazelle and Welzl's matching theorem, which we believe to be of independent interest.

\begin{lemma}\label{key}
    Let $\mathcal{F}$ be a set system on a ground set $V = \{v_1,\ldots, v_n\}$ such that $\pi^{\ast}_{\mathcal{F}}(m) \leq cm^d$ and $|\mathcal{F}|\leq cn^d$, where $d\geq 2$.  Then there is a subset $X\subset V$ such that $|X| \leq 2n^{1/2 + 1/(2d)}$, a constant $c_3 = c_3(d,c)$, and a perfect matching $M$ on $V\setminus X$, such that the following properties holds.

    \begin{enumerate}

        \item For each $\{v_i,v_j\} \in M$, $|i-j| \leq n^{1- 1/(2d)}$.

        \item For each $\{v_i,v_j\} \in M$, at most $c_3(|\mathcal{F}|/n^{1/(2d)})$ sets in $\mathcal{F}$ stabs $\{v_i,v_j\}$.

                        \item For each set $A \in \mathcal{F}$, $A$ stabs at most $c_3n^{1 - 1/d}$ members of $M$.  
 
    \end{enumerate}
\end{lemma}

\begin{proof}
    Let $\mathcal{F}$ be a set system with ground set $V$ such that $\pi^{\ast}_{\mathcal{F}}(m) = cm^d$, and let $c_3  = c_3(d,c)$ be a large constant that will be determined later.  We apply Lemma \ref{partition} to $\mathcal{F}$ with parameter $\delta = c_2\left(\frac{|\mathcal{F}|}{n^{1/(2d)}}\right)$, where $c_2 > 1$ is defined in Lemma \ref{partition}, to get a partition on $V$,

    $$\mathcal{P}:V = P_1\cup P_2\cup \cdots \cup P_r,$$

\noindent such that $r <  n^{1/2}$ and with the properties described in Lemma \ref{partition}.  We further partition each part $P_k$ into at most $n^{1/(2d)}$ smaller parts, such that two vertices $v_i,v_j$ from a resulting part satisfies $|i-j|  \leq n^{1- 1/(2d)}$.  Let

$$\mathcal{P}':V = S_1\cup S_2\cup \cdots \cup S_{t},$$

\noindent the be resulting partition, where $t < n^{1/2  + 1/(2d)}$.  Hence, for any two vertices $v_i,v_j \in S_k$ in the same part, at most $2\delta = 2c_2\left(\frac{|\mathcal{F}|}{n^{1/(2d)}}\right)$ sets in $\mathcal{F}$ stabs $\{v_i,v_j\}$, and we have $|i-j| \leq n^{1- 1/(2d)}.$   

In what follows, we construct our matching $M$ by picking the pairs $\{x_i,y_i\}$ one by one, such that for each pair $\{x_i,y_i\} \in M$, both $x_i$ and $y_i$ lie in the same part of $\mathcal{P}'$.  By setting $c_3$ sufficiently large, this will guarantee that the first two properties of the lemma to hold for $M$.   In order to achieve the third property, we apply a well-known iterative reweighting strategy (see also \cite{welzl,matousek}).

Suppose we have obtained pairs $M_i = \{\{x_1,y_1\}, \{x_2,y_2\}, \ldots, \{x_i,y_i\}\}$ for our matching, where $i \leq n/2 - n^{1/2 + 1/(2d)}$.  We define the weight $w_i(A)$ of a set $A \in \mathcal{F}$ as $2^{\kappa_i(A)}$, where $\kappa_i(A)$ is the number of matchings in $M_i$ that is stabbed by $A \in \mathcal{F}$.     Let $V_{i} = V\setminus \{x_1,y_1,\ldots, x_i,y_i\},$ and let $\mathcal{F}_{i}$ be a set system with ground set $V_{i}$, where we add the set $A\cap V_{i}$ to $\mathcal{F}_i$ with multiplicity $w_i(A)$.  Thus, $w_0(A) = 1$ for all $A \in \mathcal{F}$, and we have $V_0 = V$ and $\mathcal{F}_0 = \mathcal{F}$.

We select $\{x_{i + 1},y_{i + 1}\}$ in $V_i$ as follows.   Since $\pi^{\ast}_{\mathcal{F}_i}(m) \leq cm^d$, we apply Corollary \ref{pigeon} to $\mathcal{F}_i$ with parameter $\delta = c_2\frac{|\mathcal{F}|}{(n - 2i)^{1/d}}$, where $c_2 > 1$ is defined in Lemma \ref{partition}, to obtain a subset $U_i\subset V_i$ such that $|U_i| \geq (n - 2i)$, and each pair of vertices in $U_i$ is stabbed by at most $2\delta = 2 c_2\frac{|\mathcal{F}_i|}{(n-2i)^{1/d}}$ sets in $\mathcal{F}_i$.  Since $i \leq n/2 - n^{1/2 + 1/(2d)}$, we have $$|U_i|\geq (n -2i)  \geq 2n^{1/2 + 1/(2d)} > t.$$  By the pigeonhole principle, two vertices of $U_i$ must lie in the same part of the partition $\mathcal{P}'$.  Let $x_{i+1},y_{i+1}$ be two such vertices and set $M_{i + 1}  = M_i \cup \{x_{i + 1},y_{i + 1}\}$ and $V_{i + 1} = V_i \setminus \{x_{i + 1},y_{i + 1}\}$.  Therefore, we have

$$|\mathcal{F}_{i   + 1}| \leq |\mathcal{F}_i| + 2 c_2\frac{|\mathcal{F}_i|}{(n-2i)^{1/d}} = |\mathcal{F}_i| \left(1 + \frac{2c_2}{(n - 2i)^{1/d}}\right).$$

We repeat this process until we have obtained a matching $M$ of size $w = \lfloor n/2 - n^{1/2 + 1/(2d)}\rfloor$. At this point, we stop and set $X$ to be the remaining vertices $V\setminus \{x_1,y_1,\ldots, x_w,y_w\}$.  Therefore, we have $|X| \leq 2n^{1/2 + 1/(2d)}$ and

\begin{align*} 
|\mathcal{F}_{w}|  & \leq   \prod\limits_{i = 0}^{w  - 1}|\mathcal{F}| \left(1 + \frac{2c_2}{(n - 2i)^{1/d}}\right)   \\ 
  & \leq |\mathcal{F}|e^{\sum_{i =0 }^{w  - 1}  \frac{2c_2}{(n-2i)^{1/d}} }  \\
  & \leq |\mathcal{F}|e^{4c_2 n^{1 -1/d}}\\
& \leq e^{4c_2n^{1 - 1/d}  + \ln(|\mathcal{F}|)}\\
 & \leq  e^{8c_2n^{1 - 1/d}}.
\end{align*}

\noindent The last inequality follows from the fact that $|\mathcal{F}| \leq cn^d$.  For any set $A \in \mathcal{F}$, if $A$ stabs $\kappa$ members in $M$, then we have $2^{\kappa} \leq  e^{8c_2n^{1 - 1/d}}$, which implies that $\kappa \leq c_3n^{1 -1/d}$  for $c_3$ sufficiently large.  \end{proof}

\section{Proof of Theorem \ref{main}}\label{p1}

Let $G = (U,E)$ be a complete simple topological graph with $n+1$ vertices.  Given a subset $S \subset U$, we denote by $G[S]$ the topological subgraph of $G$ induced by the vertex set $S$.  For simplicity, we write $G[u,v] = uv$.  In order to avoid confusion between topological edges and pairs of vertices, we write $uv$ when referring to the topological edge in the plane, and write $\{u,v\}$ when referring to the pair.  By the simple condition, all 3-cycles in $G$ are non-self-intersecting, and we will sometimes refer to them as \emph{triangles} in $G$.  Also for convenience, say that a vertex $z$ lies \emph{inside} the triangle $T = G[u,v,w]$, if $z$ lies in the bounded region enclosed by $G[u,v,w]$.  Moreover, we say that the edge $xy\in E(G)$ lies \emph{inside} of triangle $G[u,v,w]$ if $x$ and $y$ lie inside $G[u,v,w]$, and $xy$ does not cross  $G[u,v,w]$.  See Figure \ref{figint}.

Notice that the edges of $G$ divide the plane into several cells (regions), one of which is unbounded.  We can assume that there is a vertex $v_0 \in U$ such that $v_0$ lies on the boundary of the unbounded cell.  Indeed, otherwise we can project $G$ onto a sphere, then choose an arbitrary vertex $v_0$ and then project $G$ back to the plane such that $v_0$ lies on the boundary of the unbounded cell, and moreover two edges cross in the new drawing if and only if they crossed  in the original drawing.

 Consider the topological edges emanating out from $v_0$, and label their endpoints $v_1,\dots , v_{n}$ in
counter-clockwise order.  Thus, for $i < j$, the sequence $(v_0,v_i,v_j)$ appears in counter-clockwise order along the non-self-intersecting 3-cycle $G[v_0,v_i,v_j]$. Set $V = U \setminus v_0$.  For $1\leq i < j \leq n$, let $I_{i,j} = \{v_k \in V: i < k < j\}$. 

We define the set system $\mathcal{F}$ with ground set $V$ as follows.  For each pair $v_i,v_j \in V$ with $i < j$, we define the set $T_{i,j}$ to be the set of vertices in $V$ that lie inside the triangle $G[v_0,v_i,v_j]$.  See Figure~\ref{set1} for a small example. Then set

 \begin{figure}
  \centering
  \subfigure[{Edge $xy$ inside triangle $G[u,v,w]$.}]{\label{figint}\includegraphics[width=0.18\textwidth]{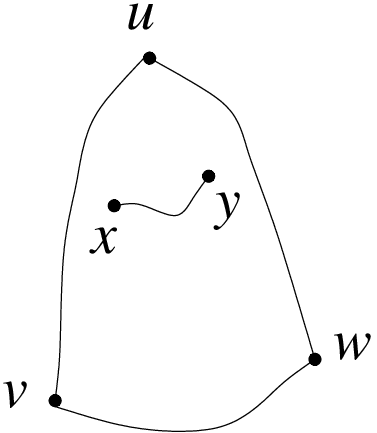}}\hspace{3cm}
    \subfigure[$T_{3,6} = \{v_{1},v_4\}$.]{\label{set1}\includegraphics[width=0.3\textwidth]{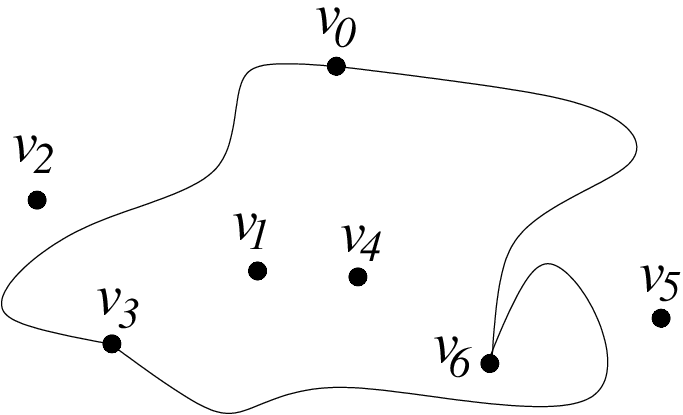}}
 \caption{Triangles in $G$.}
\end{figure}

$$\mathcal{F} = \{T_{i,j}: 1\leq i < j \leq n\}.$$

In \cite{suk}, the author proved the following lemma, and for sake of completeness, we include its short proof here.

\begin{lemma}[\cite{suk}]
\label{s1}
Let $\mathcal{F}$ be the set system with ground set $V$ defined as above.  Then $\pi^*_{\mathcal{F}}(m) \leq 5m^2$.

\end{lemma}

\begin{proof}
In order to obtain an upper bound on $\pi^*_{\mathcal{F}}(m)$, we simply need to bound the maximum number of regions for which $m$ triangles in $G$ partitions the plane into.  We proceed by induction on $m$.  The base case when $m = 1$ is trivial.  Now assume that the statement holds up to $m-1$.  Notice that any set of $m-1$ triangles consists of at most $3(m-1)$ topological edges of $G$.  Therefore, when we add the $m$th triangle $T = G[v_0,v_i,v_j]$, each topological edge in $T$ creates at most $3(m-1)$ new crossing points since $G$ is simple.  Hence, triangle $T$ creates at most $9(m-1)$ new regions in the arrangement.  By the induction hypothesis, we have at most

$$5(m-1)^2 + 9(m-1) \leq 5m^2$$

\noindent cells created by the $m$ triangles.\end{proof}

Since the set system $\mathcal{F}$ satisfies $\pi^{\ast}_{\mathcal{F}}(m) \leq 5 m^2$ and $|\mathcal{F}| \leq n^2/2$, we can apply Lemma \ref{key} to $\mathcal{F}$ to obtain a subset $X\subset V$ such that $|X|\leq 2n^{3/4}$, a perfect matching $M_1$ on $V\setminus X$, and an absolute constant $c_3$ such that the following holds.

   \begin{enumerate}

        \item For each $\{v_i,v_j\} \in M_1$, $|i-j| \leq n^{3/4}$.

        \item For each $\{v_i,v_j\} \in M_1$, at most $c_3(|\mathcal{F}|/n^{1/4})$ sets in $\mathcal{F}$ stabs $\{v_i,v_j\}$.

                        \item For each set $A \in \mathcal{F}$, $A$ stabs at most $c_3n^{1/2}$ members of $M_1$.  
 
    \end{enumerate}

Let $\Gamma_1$ be a directed auxiliary graph such that $V(\Gamma_1) = M_1$, and for $\{v_i,v_j\},\{v_k,v_{\ell}\}\in V(\Gamma_1)$, we have a directed edge from $\{v_i,v_j\}$ to   $\{v_k,v_{\ell}\}$ if and only if $T_{i,j}$ stabs the pair $\{v_k,v_{\ell}\}$.  By the third property above, $\Gamma_1$ has at most $|M_1|c_3n^{1/2} \leq c_3n^{3/2}$ directed edges, and therefore, the sum of the in-degree of each vertex in $\Gamma_1$ is at most $c_3n^{3/2}$.  Hence, the number of vertices in $\Gamma_1$ whose in-degree at least $n^{3/4}$ is at most

$$\frac{c_3n^{3/2}}{n^{3/4}} = c_3n^{3/4}.$$

Let $M_2\subset M_1$ be the members in $M_1$ whose in-degree is less than $n^{3/4}$ in $\Gamma_1$.  Hence, $|M_1\setminus M_2|\leq c_3n^{3/4}$, and

$$|M_2| \geq \frac{n}{2} - n^{3/4} - c_3n^{3/4}.$$

\noindent For integers $1 \leq i < j \leq n$, let $\phi_{M_2}(v_i,v_j)$ denote the number of matchings $\{v_k,v_{\ell}\} \in M_2$ such that both $v_k,v_{\ell}$ lie inside triangle $G[v_0,v_i,v_j]$.  Set $\{v_x,v_y\} \in M_2$ to be the matching such that $\phi_{M_2}(v_x,v_y) \leq \phi_{M_2}(v_i,v_j)$,  for any $\{v_i,v_j\} \in M_2$.  We now make the following claim.

   \begin{claim}
        The topological edge $v_xv_y$ crosses at most $c_4n^{7/4}$ other edges in $G$, where $c_4$ is a sufficiently large absolute constant. 
   \end{claim}

   \begin{proof}
 
 We define subsets $E_0,E_1,E_2,E_3,E_4 \subset E(G)$ as follows.  Let $E_0$ be the edges in $G$ incident to $v_0$.  Let $E_1 \subset E(G)$ be the set of edges in $G$ that has an endpoint in $X$.   Let $E_2$ be the set of edges $v_iv_j$ in $G$ with an endpoint in $I_{x,y} = \{v_k \in V: x < k < y\}$.  Let $E_3$ be the set of edges $v_iv_j$ such that $v_iv_j$ crosses $v_xv_y$, and the set $T_{i,j}$ stabs the pair $\{v_x,v_y\}$.   Let $E_4$ be the set of edges $v_iv_j$ that crosses $v_xv_y$, such that $v_iv_j$ is not in $E_0,E_1$, $E_2$, nor $E_3$.  It is easy to see that for any edge $v_iv_j \in E(G)$ that crosses $v_xv_y$, we have $v_iv_j \in E_0\cup E_1 \cup\cdots\cup E_4$. 

Clearly, we have
\begin{equation}\label{e0}|E_0| \leq n.\end{equation}

\noindent Since $|X| \leq 2n^{3/4}$, we have
\begin{equation}\label{e1}|E_1| \leq 2n^{7/4}.\end{equation} 
 
\noindent By the first property of our matching,  we have $|x - y| < n^{3/4}$.  Therefore, we have \begin{equation}\label{e2}
|E_2| \leq n^{7/4}.\end{equation}

\noindent Since $\{v_x,v_y\} \in M_2$, at most $c_3|\mathcal{F}|/n^{1/4} \leq c_3n^{7/4}$ sets in $\mathcal{F}$ stabs $\{v_x,v_y\}$.  Hence, we have \begin{equation}\label{e3}|E_3| \leq c_3n^{7/4}.\end{equation}

 \begin{figure}
  \centering
\subfigure[$x < y < i < j$.]{\label{fignote42}\includegraphics[width=0.22\textwidth]{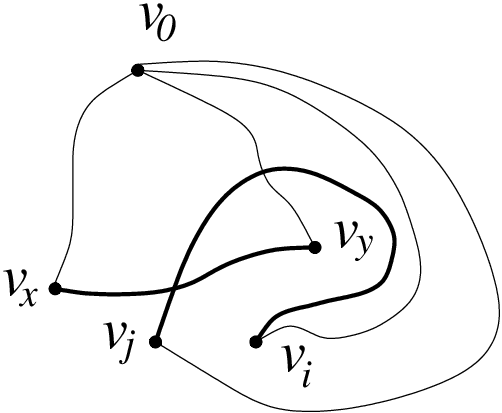}}  \hspace{.5cm}
  \subfigure[$i < x < y < j$.]{\label{fignote4}\includegraphics[width=0.22\textwidth]{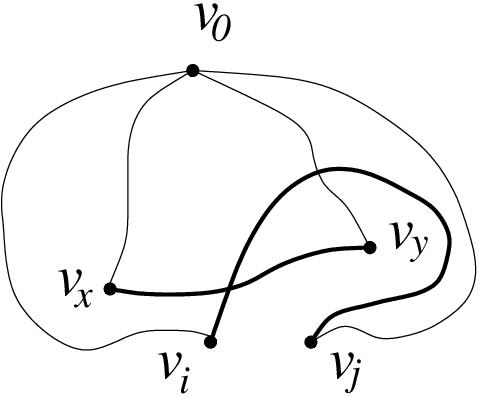}}\hspace{.5cm}
  \subfigure[$i < x < y < j$.]{\label{fignote42a}\includegraphics[width=0.22\textwidth]{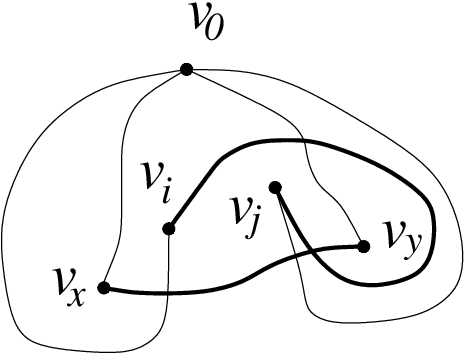}}\hspace{.5cm}
  \subfigure[$x < y < i < j$.]{\label{fignote42b}\includegraphics[width=0.22\textwidth]{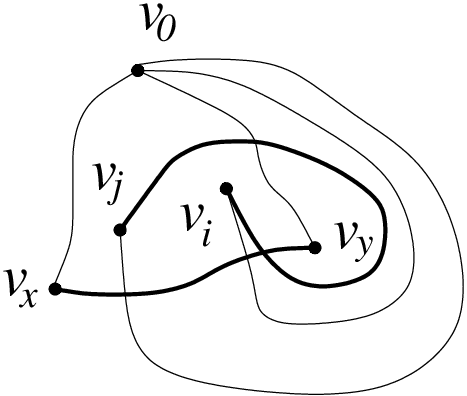}}
 \caption{Observation \ref{obs1} with $v_i,v_j \in U_1$ or $v_i,v_j \in U_2$.}\label{note4f}
\end{figure}

\noindent Thus, it remains to bound $|E_4|$.  Let

 $$U_1 = \{v_k \in V\setminus I_{x,y}: v_0v_k \textnormal{ crosses } v_xv_y\}$$ and $$U_2 = \{v_k \in V\setminus I_{x,y}: v_0v_k \textnormal{ is disjoint to }v_xv_y\}.$$

\noindent We now make the following observation.

\begin{observation}\label{obs1}
    If $v_iv_j \in E_4$, then $v_iv_j$ has one endpoint in $U_1$ and the other in $U_2$.
\end{observation}

\begin{proof}
    For sake of contradiction, suppose that $v_i,v_j \in U_2$ such that $i < j$. Since $v_iv_j$ crosses $v_xv_y$, and since $G$ is simple, $v_iv_j$ must cross either $v_0v_x$ or $v_0v_y$, but not both.  Without loss of generality, we can assume that $v_iv_j$ crosses $v_0v_y$.  Since $v_iv_j \in E_4$, we have that $v_i,v_j \not\in I_{x,y}$.  Hence, if we have $i < j < x$ or $y < i < j$, then the set $T_{i,j}$ contains $v_y$ but not $v_x$.  Therefore, $T_{i,j}$ stabs $\{v_i,v_j\}$, and we have $v_iv_j \not\in E_4$ which is a contradiction.  See Figure \ref{fignote42}.  If we have $i < x < y < j$, then the set $T_{i,j}$ contains $v_x$ but not $v_y$. See Figure \ref{fignote4}.  Hence, $T_{i,j}$ stabs $\{v_x,v_y\}$ which implies that $v_iv_j \not\in E_4$, contradiction.  A symmetric argument follows if $v_i,v_j \in U_1$.  See Figures \ref{fignote42a} and \ref{fignote42b}.    \end{proof}

For sake of contradiction, let us suppose that $|E_4| >   (c_4/2)n^{7/4}$, where $c_4$ is a large constant that will be determined later.  Since $|U_2|\leq n$, the observation above implies that $|U_1| > (c_4/2)n^{3/4}.$  In what follows, we will find another matching $\{v_{x'},v_{y'}\} \in M_2$ such that $\phi_{M_2}(v_{x'},v_{y'
}) < \phi_{M_2}(v_{x},v_{y
})$, which will be our contradiction.

\begin{claim}
There is a subset $M_4\subset M_2$, where

    $$|M_4| \geq \frac{c_4n^{3/4}}{16},$$
    
\noindent  such that for each matching  $\{v_i,v_j\} \in M_4$, we have $v_i,v_j \in U_1$ and the topological edge $v_iv_j$ lies completely inside of triangle $G=[v_0,v_x,v_y]$.

\end{claim}

 \begin{proof}
Let $S_1,\ldots, S_4 \subset U_1$, such that $S_1 = U_1\cap X$, 

$$S_2 = \{v_i \in U_1: \{v_i,v_j\} \in M_1\setminus M_2\},$$

$$S_3 = \{v_i \in U_1: \{v_i,v_j\} \in M_2, v_j \in I_{x,y}\},$$

$$S_4 = \{v_i \in U_1: \{v_i,v_j\} \in M_2, v_j \in U_2\}.$$

Clearly we have $|S_1| \leq 2n^{3/4}$ since $|X| \leq 2n^{3/4}$.  Moreover, at most $2|M_1\setminus M_2| \leq 2c_3n^{3/4}$ vertices in $U_1$ are endpoints of a matching in $M_1$ but not in $M_2$.  Thus $|S_2| \leq  2c_3n^{3/4}$.  Since $|I_{x,y}| \leq n^{3/4}$, we have $|S_3| \leq n^{3/4}$.  Finally, if $v_i\in U_1$ and $v_j\in U_2$, the set $T_{x,y}$ stabs $\{v_i,v_j\}.$  By the third property of the matching $M_1$, we have $|S_4| \leq c_3n^{1/2}.$

Let $S_5 = U_1\setminus (S_1\cup\cdots \cup S_4)$.  By setting $c_4$ sufficiently large, we have

\begin{align*} 
|S_5| &\geq   |U_1| -  2n^{3/4} - 2c_3n^{3/4} - n^{3/4} - c_3n^{1/2} \\ 
  &\geq  (c_4/4)n^{3/4}. \\ 
\end{align*}

Thus, we have a perfect matching $M_3$ on $S_5$ such that $M_3\subset M_2$.  Finally, let us count the number of matchings $\{v_i,v_j\} \in M_3\subset M_1$ such that the topological edge $v_iv_j$ does lie completely inside triangle $G[v_0,v_x,v_y]$.  If $\{v_i,v_j\}\in M_3$ and $v_iv_j$ does not lie completely inside of $G[v_0,v_x,v_y]$, then $v_iv_j$ must cross exactly two edges of the triangle $G[v_0,v_x,v_y]$.  Moreover, the two edges cannot be $v_0v_x$ and $v_0v_y$.  See Figures \ref{figobs1} and \ref{figobs2}.  Hence, $v_iv_j$ must then cross $v_xv_y$.  As argued in Observation~\ref{obs1}, the set $T_{i,j}$ stabs $\{v_x,v_y\}$.  Since the in-degree of $\{v_x,v_y\}\in M_2$ is less than $n^{3/4}$ in the directed graph $\Gamma_1$, there are at most $n^{3/4}$ such matchings.  We delete these matchings, leaving us with a matching $M_4\subset M_3$, where

 \begin{figure}
  \centering
\subfigure[$x < y < i < j$.]{\label{figobs1}\includegraphics[width=0.25\textwidth]{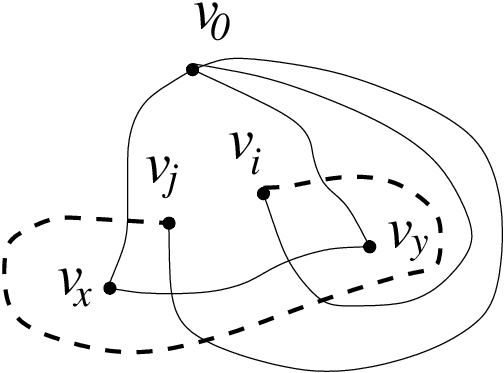}}  \hspace{2cm}
  \subfigure[$i < x < y < j$.]{\label{figobs2}\includegraphics[width=0.25\textwidth]{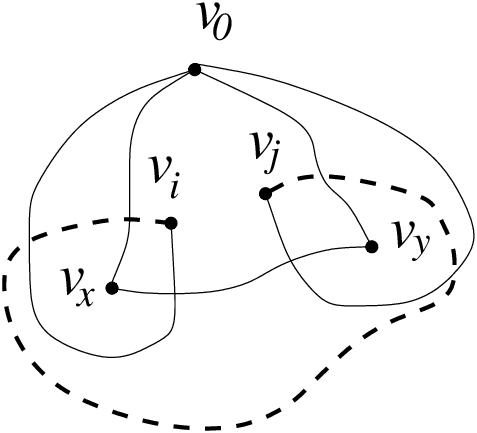}}
 \caption{Edge $v_iv_j$ violating the simple condition by crossing both $v_0v_x$ and $v_0v_y$.}\label{figdouble}
\end{figure}

\begin{align*} 
|M_4| &\geq \frac{|S_5|}{2} - n^{3/4} \\
    &\geq (c_4/8)n^{3/4} - n^{3/4}\\
    & \geq (c_4/16)n^{3/4},
\end{align*}

\noindent such that for each $\{v_i,v_j\}\in M_4$, $v_i,v_j \in U_1$, and the topological edge $v_iv_j$ lies completely inside of triangle $G[v_0,v_x,v_y]$.   
 \end{proof}

\noindent Fix a matching $\{v_i,v_j\} \in M_4$.  Then notice that for any vertex $v_k \in U_1$ such that $v_k \not\in I_{x,y}\cup I_{i,j}$, if $v_k$ lies inside of triangle $G[v_0,v_i,v_j]$, then $v_k$ lies inside of triangle $G[v_0,v_x,v_y]$. See Figure \ref{figkeep1}. We create a second auxiliary directed graph $\Gamma_4$, where $V(\Gamma_4) = M_4$, and we have a directed edge from $\{v_k,v_{\ell}\}$ to $\{v_i,v_{j}\}$ if and only if $v_k$ and $v_{\ell}$ lie inside of triangle $G[v_0,v_i,v_{j}]$.

\begin{observation}
    Let $\{v_i,v_j\}, \{v_k,v_{\ell}\} \in V(\Gamma_4)$. Suppose $v_i,v_j \not\in I_{k,\ell}$ and $v_k,v_{\ell} \not\in I_{i,j}$.  Then there is at most one directed edge between $\{v_i,v_j\}$ and $\{v_i,v_{\ell}\}$ in $\Gamma_4$.  
\end{observation}

\begin{proof}
    Without loss of generality, we can assume $i < j < k< \ell$.  For sake of contradiction, suppose we have two directed edges (both directions) between $\{v_i,v_j\}$ and $\{v_i,v_{\ell}\}$ in $\Gamma_4$.  Hence, $v_k,v_{\ell}$ both lie inside of triangle $G[v_0,v_i,v_j]$.  If the edge $v_kv_{\ell}$ lies inside of $G[v_0,v_i,v_j]$, then we only have one directed edge between between $\{v_i,v_j\}$ and $\{v_i,v_{\ell}\}$ in $\Gamma_4$, which is a contradiction.  Therefore, we can assume that $v_kv_{\ell}$ crosses $G[v_0,v_i,v_j]$. As argued in Observation \ref{obs1}, edge $v_kv_{\ell}$ must cross $v_iv_j$, and either $v_0v_i$ or $v_0v_j$ but not both.  However, this implies that the triangle $G[v_0,v_k,v_{\ell}]$ does not contain both $v_i$ and $v_j$, which is a contradiction.  See Figure \ref{fignote42b}.
\end{proof}

Let us remark that there can be pairs of vertices $\{v_i,v_j\}$ and $\{v_i,v_{\ell}\}$ in $\Gamma_4$ with two directed edges between them in $\Gamma_4$.  However, this would only happen if $v_i,v_j \in I_{k,\ell}$ or $v_k,v_{\ell} \in I_{i,j}$ by the observation above.   See Figure~\ref{figg42}.  Moreover, since $|I_{i,j}|,|I_{k,\ell}| \leq n^{3/4}$, the number of such pairs is at most $|V(\Gamma_4)|n^{3/4}$.  Since $|V(\Gamma_4)| = |M_4| \geq (c_4/16)n^{3/4}$, by setting $c_4$ sufficiently large, the total number of directed edges in $\Gamma_4$ is at most

 \begin{figure}
  \centering
    \subfigure[{$v_k$ inside $G[v_0,v_x,v_y]$ and $G[v_0,v_i,v_j]$.}]{\label{figkeep1}\includegraphics[width=0.23\textwidth]{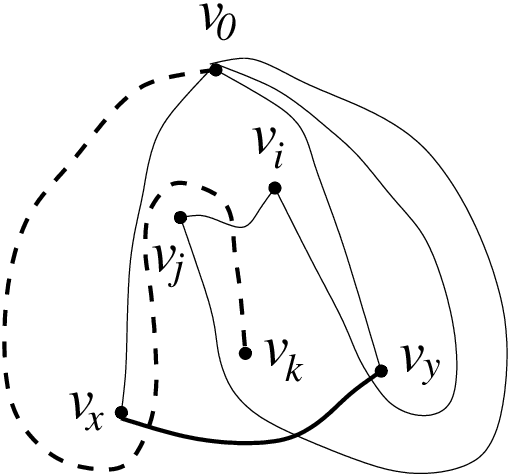}}\hspace{1cm}
\subfigure[$v_k,v_{\ell} \in I_{i,j}$]{\label{figg42}\includegraphics[width=0.25\textwidth]{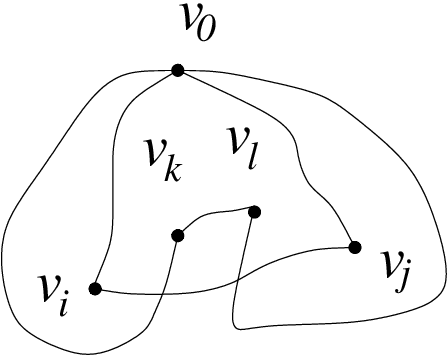}}\hspace{1cm}
  \subfigure[{Matching $\{v_i,v_j\}$ not in triangle $G[v_0,v_{x'},v_{y'}].$}]{\label{figexm1}\includegraphics[width=0.25\textwidth]{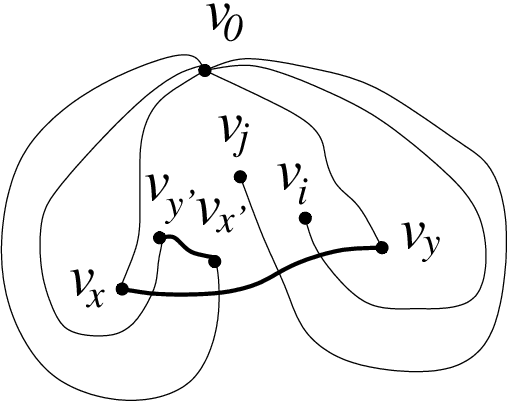}}
 \caption{Directed edges between $\{v_i,v_j\}$ and $\{v_{k},v_{\ell}\}$ in $\Gamma_4$.}\label{figg4}
\end{figure}

$$\binom{|V(\Gamma_4)|}{2} + |V(\Gamma_4)|n^{3/4} \leq \frac{3}{2}\binom{|V(\Gamma_4)|}{2}.$$

\noindent Hence, the sum of the in-degree of each vertex in $\Gamma_4$ is at most $\frac{3}{2}\binom{|V(\Gamma_4)|}{2}.$  By averaging, there is a vertex in $\Gamma_4$ whose in-degree is at most $(3/4)|V(\Gamma_4)|$.  Let us fix such a matching $\{v_{x'},v_{y'}\} \in V(\Gamma_4) = M_4$.  Hence, there are at least 

$$\frac{|V(\Gamma_4)|}{4} = \frac{|M_4|}{4} \geq \frac{c_4}{64}n^{3/4}$$ 

\noindent matchings in $M_4$, and thus, in $M_2$, that lies in the triangle $G[v_0,v_x,v_y]$, but not in $G[v_0,v_{x'},v_{y'}].$  See Figure \ref{figexm1}.  On the other hand, there are not too many matchings in $M_2$ that lies in triangle $G[v_0,v_{x'},v_{y'}]$, but not in triangle $G[v_0,v_x,v_y]$.  By the arguments above and Figure \ref{figkeep1}, each such matching must have an endpoint in $I_{x,y}\cup I_{x',y'}$.  Since $|I_{x,y}|,|I_{x',y'}|\leq n^{3/4}$, there are at most $2n^{3/4}$ such matchings in $M_2$.

Therefore, for $c_4$ sufficiently large, we have

\begin{align*} 
\phi_{M_2}(v_{x'},v_{y'}) &\leq  \phi_{M_2}(v_x,v_y)  - |M_4|/4 + 2n^{3/4}  \\
    &\leq  \phi_{M_2}(v_x,v_y)  - (c_4/64)n^{3/4} + 2n^{3/4}   \\
    &<  \phi_{M_2}(v_x,v_y).
\end{align*}

\noindent This contradicts the minimality of $\phi_{M_2}(v_x,v_y)$.  Hence

\begin{equation}\label{e4}
    |E_4| \leq (c_4/2)n^{3/4}.
\end{equation}

\noindent  Combining (\ref{e0}), (\ref{e1}), (\ref{e2}), (\ref{e3}), (\ref{e4}), we conclude that the number of edges in $G$ crossing the topological edge $v_xv_y$ is at most 

$$|E_0| + |E_1|  + |E_2| + |E_3| + |E_4| \leq n + 2n^{7/4} + n^{7/4} + c_3n^{7/4} + (c_4/2)n^{7/4} \leq  c_4n^{7/4},$$

\noindent for sufficiently large $c_4$.  This proves our claim, and completes the proof of Theorem \ref{main}.\end{proof}

\section{Concluding remarks}

Theorem \ref{main} answers one of the questions raised by the author and Zeng in \cite{zeng}.  It would be interesting if one could improve this bound further.

\begin{conjecture}
    For $n\geq 2$, we have $h(n) = \Theta(n^{3/2})$.  
\end{conjecture}

 It is not even known whether or not every complete $n$-vertex simple topological graph contains $\Omega(n)$ pairwise disjoint edges.  The best known bound is due to Aicholzer et al.~\cite{aich}, who showed that one can always find $\Omega(n^{1/2})$ pairwise disjoint edges in a complete $n$-vertex simple topological graph.  See also \cite{suk,fulek,ruiz} for slightly weaker polynomial bounds.
   A very strong conjecture due to Rafla \cite{raf} states that one can always find a noncrossing Hamiltonian cycle, which has been verified for $n \leq 9$ (see \cite{ab}).

 Unfortunately, the proofs in \cite{suk,fulek,ruiz,aich} breakdown when trying to find a noncrossing path rather than disjoint edges.   Recently, Aichholzer et al. \cite{aich}, and independently, the author and Zeng \cite{zeng} showed that one can always find a noncrossing path of length $\Omega(\frac{\log n}{\log\log n})$ in a complete $n$-vertex simple topological graph.   The following was conjectured by the author and Zeng.

\begin{conjecture}[\cite{zeng}]
    There is an absolute constant $\varepsilon > 0$, such that every complete $n$-vertex simple topological graph contains a noncrossing path on $n^{\varepsilon}$ vertices.
\end{conjecture}

\end{document}